\documentclass{article}
\usepackage{amssymb}
\usepackage{amsfonts}
\usepackage{amsmath}
\usepackage{graphicx}

\newtheorem{theorem}{Theorem}

\newtheorem{corollary}[theorem]{Corollary}

\newtheorem{definition}[theorem]{Definition}
\newtheorem{example}[theorem]{Example}

\newtheorem{lemma}[theorem]{Lemma}

\newenvironment{proof}[1][Proof]{\noindent\textbf{#1.} }{\ \rule{0.5em}{0.5em}}

\begin{document}

\title{On the characteristic polynomial of Laplacian Matrices of Caterpillars}

\author{\textbf{Domingos Moreira Cardoso}\\
{\small CIDMA-Center for Research and Development in Mathematics and Applications,}\\
{\small Department of Mathematics, University of Aveiro, 3810-193 Aveiro, Portugal}\\
{\small \texttt{dcardoso@ua.pt}} \and \textbf{Maria A. A. de Freitas}\\
{\small Instituto de Matem\'{a}tica and COPPE/Produ\c{c}\~{a}o, }\\
{\small Universidade Federal de Rio de Janeiro,}\\
{\small \texttt{maguieiras@im.ufrj.br}} \and \textbf{Enide Andrade Martins} \\
{\small CIDMA-Center for Research and Development in Mathematics and Applications,}\\
{\small Department of Mathematics, University of Aveiro, 3810-193 Aveiro, Portugal}\\
{\small \texttt{enide@ua.pt}} \and \textbf{Mar\'{\i}a Robbiano} and \textbf{Bernardo San Mart\'{\i}n}\\
{\small Departamento de Matem\'{a}ticas }\\
{\small Universidad Cat\'{o}lica del Norte}\\
{\small Av. Angamos 0610 Antofagasta, Chile}\\
{\small \texttt{mrobbiano@ucn.cl} \& \texttt{sanmarti@ucn.cl}}}

\maketitle

\begin{abstract}
The characteristic polynomials of the adjacency matrix of line graphs of caterpillars and then the characteristic
polynomials of their Laplacian or signless Laplacian matrices are characterized, using recursive formulas. Furthermore,
the obtained results are applied on the determination of upper and lower bounds on the algebraic connectivity of these graphs.
\end{abstract}

\baselineskip=0.30in

\section{Introduction}

In this paper we consider undirected simple graphs $G$ (that is, without loops and parallel edges), simply called graphs.
The vertex set of $G$ is denoted $V(G)$ and its cardinality is called the order of $G$. The edge set of $G$ is denoted
$E(G)$. Two vertices $x,y \in V(G)$ are \textit{adjacent} when they are connected by an edge $xy \in E(G)$.
A \textit{complete} graph of order $n$, $K_n$, is a graph where each pair of vertices are adjacent. The graph $K_1$
is the trivial graph (with just one vertex) and in this text $K_0$ denotes the graph without vertices. The
\textit{neighbors} of a vertex are the vertices adjacent to it. The set of neighbors of a vertex $v\in V\left( G\right)$
(the \textit{neighborhood} of $v$) is denoted  $N_{G}(v)$ and its cardinality, that is, the \textit{degree} by
$d\left(v\right)$. A $p$-regular graph is a graph where each vertex has degree $p$. A \textit{pendant} vertex of $G$ is a
vertex of degree $1$.

The \textit{adjacency matrix} of the graph $G$ is the $n \times n$ symmetric matrix $A\left( G\right) =\left( a_{ij}\right)$
where $a_{ij}=1$ if $ij\in E(G)$ and $a_{ij}=0$ otherwise. The \textit{Laplacian} (\textit{signless Laplacian}) matrix of $G$
is the matrix $L(G)=D(G)-A(G)$ ($Q(G)=D(G)+A(G)$), where $D(G)$ is the $n\times n$ diagonal matrix of vertex degrees of $G$.
Since all these matrices are real and symmetric, their eigenvalues are all real (nonnegative in the Laplacian and signless
Laplacian cases). The \textit{spectrum} of a square matrix $M$ of order $n$, that is, the multiset of its eigenvalues,
$\lambda_1(M), \ldots, \lambda_n(M)$, is denoted $\sigma _{M}$. In the particular cases of $L(G)$ and $Q(G)$, their spectra
are denoted by $\sigma _{L}(G)$ and $\sigma _{Q}(G)$, respectively. Throughout the paper
$\sigma_{L}(G)=\{\mu _{1}^{[i_{1}]},\ldots ,\mu _{p}^{[i_{p}]}\}$ ($\sigma_{Q}(G)=\{q_{1}^{[k_{1}]},\ldots ,q_{r}^{[k_{r}]}\}$)
means that $\mu _{j}$ ($q_{j}$) is a Laplacian (signless Laplacian) eigenvalue with multiplicity $i_{j}$ ($k_{l}$),
for $j=1,\ldots ,p$ ($l=1,\ldots ,r$). As usually, we denote the eigenvalues of $L\left( G\right) $ ($Q\left( G\right) $)
in non increasing order by $\mu _{1}(G)\ge \cdots \ge \mu _{n}(G)$ ($q_{1}(G)\ge \cdots \ge q_{n}(G)$). For details on
the spectral properties of $L(G)$ and $Q(G)$ we refer the reader to \cite{lapl2, Grone, lapl1, Merris} and
\cite{Domingos, lapl4}, respectively. A \textit{path} with $k$ vertices, $P_k$, of $G$ is a sequence of $k$
vertices $v_1, \ldots, v_k$, such that $v_iv_{i+1} \in E(G)$ for $i \in \{1, \ldots, k-1\}$ and all vertices
are distinct except eventually $v_1$ and $v_k$. When $v_1=v_k$, $P_k$ is a closed path which is called \textit{cycle}.
The \textit{length} of a path $P_k$, is the number of its edges, that is, $k-1$. A graph is \textit{connected}
when there is a path between each pair of vertices. A \textit{tree} is a connected graph without cycles.
For $k \ge 0$, a \textit{star} with $k+1$ vertices, $S_k$, is a tree with a central vertex with degree $k$
and all remaining vertices are pendant. A \textit{caterpillar} is a tree of order $n \ge 5$ (notice that a
tree of order less than $5$ is a path or a star) such that removing all the pendant vertices produces a path
with at least two vertices. In particular, the caterpillar $T(q_1, \ldots, q_k)$ is obtained from a path $P_k,$
with $k \ge 2$, attaching the central vertex of the star $S_{q_i}$ ($1 \le i \le k$) to the $i$th vertex of the
path $P_k$. Then, the order of the caterpillar is $n = q_1 + \cdots + q_k + k$. For a graph $G$, the \textit{line}
graph of $G$, $\mathcal{L}(G)$, is a graph with vertex set $V \left( \mathcal{L}(G)\right)=E\left( G\right)$ and
edge set
$$
E\left( \mathcal{L}(G)\right) = \left\{e_ie_j: e_i, e_j \in E(G) \text{ and these edges have a common vertex in } G \right\}.
$$

Let $I(G)$ be the (vertex-edge) \textit{incidence} matrix of the graph $G$ defined as the $n\times m$ matrix whose
$(i,j)$-entry is $1$ if the vertex $v_{i}$ is an end-vertex of the edge $e_{j}$ and $0$ otherwise. Consider the
following well known identities:
\begin{eqnarray}
I(G)\,I(G)^{t} &=&A(G)+D(G)=Q(G)  \label{ky} \\
I(G)^{t}\,I(G) &=&2\,\mathrm{I}_{m}+A_{\mathcal{L}}\left( G\right), \label{sec}
\end{eqnarray}%
where $\mathrm{I}_{m}$ denotes the identity matrix of order $m$ and $A_{\mathcal{L}}\left( G\right)$ is the adjacency
matrix of the line graph ${\cal L}(G)$ of the graph $G$ (see, for instance, \cite{inci}). Since when $A$ and $B$ are
matrices of orders  $t\times s$ and $s\times t$, respectively, $AB$ and $BA$ have the same nonzero eigenvalues
\cite{HornJohnson88}, we may conclude that the nonzero eigenvalues of $Q(G)$ and $A_{\mathcal{L}}\left( G\right)$ are
shifted by $2$. On the other hand, as it is well known, the spectra of $L(G)$ and $Q(G)$ coincide if and only if $G$
is a bipartite graph (see \cite{lapl2, lapl1}). Therefore, the nonzero eigenvalues of $Q(G)$ and $L(G)$ can be obtained
from $A_{\mathcal{L}}\left( G\right)$, when $G$ is bipartite, as it is the case of caterpillar
graphs. Assuming that $G$ is a connected graph, we may synthesize all of these conclusions in the following result.

\begin{theorem}\label{basic_result}
If $G$ is a bipartite connected graph, then
$$
\sigma(Q(G)) = \sigma(L(G)) = \left(\sigma(A_{\cal L}(G))+2\right)^+ \cup \{0\},
$$
where $\sigma(A_{\cal L}(G))+2$ denotes the spectrum of $A_{\cal L}(G)$ with each eigenvalue added by $2$ and
$\left(\sigma(A_{\cal L}(G))+2\right)^+$ is the multiset of positive elements of $\sigma(A_{\cal L}(G))+2$.
\end{theorem}

It should be noted that, if $G$ has no cycles, $-2$ is not an eigenvalue of $A_{\cal L}(G)$ and then
$\left(\sigma(A_{\cal L}(G))+2\right)^+=\sigma(A_{\cal L}(G))+2.$

Some results on the Laplacian and signless Laplacian eigenvalues of a caterpillar were obtained in
\cite{rojo_et_al2010-1, rojo_et_al2010-2, rojo2011}. In particular, in \cite{rojo_et_al2010-1},
the caterpillars with maximum and minimum algebraic connectivity (concept introduced in \cite{f73}) were found and the algebraic
connectivity of symmetric caterpillars are characterized by the  smallest eigenvalue of a $2\times2$ block tridiagonal matrix;
in \cite{rojo_et_al2010-2} the caterpillars of type $T(1, \ldots,1, q_i, 1, \ldots, 1, q_{k+1-i},1, \ldots, 1)$, where
$q_i \le q_{k+1-i}$ and $1 \le i \le \lfloor \frac{k}{2}\rfloor$
were studied and in \cite{rojo2011} the eigenvalues and the energy of $\mathcal{L}\left(T(q_1, \ldots, q_k)\right)$ are characterized, taking into account
that this line graph becomes a sequence of complete graphs $K_{q_1+1}, K_{q_2+2}, \ldots, K_{q_{k-1}+2}, K_{q_k+1}$,
such that two consecutive complete graphs have in common exactly one vertex. Furthermore, explicit formulas are given
when $q_1 = \cdots = q_k$.

In this paper, a recursive formula for the characteristic polynomial of the Laplacian matrix of a caterpillar $T(q_1, \ldots, q_k)$
is obtained by means of a recursive formula for the characteristic polynomial of the adjacency matrix of the line graph
$\mathcal{L}\left( T\left(q_1, \ldots, q_k\right) \right),$ using the concept of $H$-join \cite{CFMR2013} (generalized composition
in \cite{Schwenk74}). Additionally, some upper and lower bounds on the algebraic connectivity of caterpillars are
introduced and a few computational experiments are presented.

\section{Characterization of Laplacian eigenvalues of caterpillars using $H$-joins}

We start this section with the definition of $H$-Join of a family of $k$ graphs.

\begin{definition}\cite{CFMR2013}
Consider a family of $k$ graphs, $\mathcal{F}=\{G_1, \ldots, G_k\}$, where each graph $G_j$ has order $n_j$, for $j=1, \ldots, k,$
and a graph $H$ such that $V(H)=\{1, \ldots, k\}$. Each vertex $j \in V(H)$ is assigned to the graph $G_j \in \mathcal{F}$.
The $H$-join of $G_{1},\ldots ,G_k$ is the graph $G=\bigvee_{H}{\{G_{j}: j \in V(H)\}}$ such that $V(G)=\bigcup_{j=1}^{k}{V(G_j)}$
and
\begin{equation*}
E(G)=\left( \bigcup_{j=1}^{k}{E(G_j)}\right) \cup \left( \bigcup_{rs\in E(H)}{\{uv:u\in V(G_{r}),v\in V(G_{s})\}}\right) .
\end{equation*}
\end{definition}

Notice that when $H=K_2$, the $H$-join of $G_1$ and $G_2$ is the usual join operation $G_1 \vee G_2$.

Consider a caterpillar $T(q_1, \ldots, q_k)$. From the above definition, it is immediate that
${\cal L}(T(q_1, \ldots, q_k))$ is the $H$-join of the family of graphs
$$
F=\{K_{q_1}, K_1, K_{q_2}, K_1, \ldots, K_{q_{k-1}}, K_1, K_{q_k}\},
$$
where, defining the function $\delta(q) = \left\{\begin{array}{ll}
                                                        1, & \hbox{if } q>0\\
                                                        0, & \hbox{otherwise}
                                                 \end{array}
                                          \right.,$ the graph $H$ is the line graph of the caterpillar $T(\delta(q_1), \ldots, \delta(q_k))$.
Therefore, we have the following result.

\begin{theorem}
Consider a caterpillar $T(q_1, \ldots, q_k)$. Then
$$
{\cal L}(T(q_1, \ldots, q_k)) = \bigvee_{H}{\mathcal{F}},
$$
where $\mathcal{F}=\{G_{v_{q_1}}=K_{q_1}, G_{v_{12}}=K_1, G_{v_{q_2}}=K_{q_2}, G_{v_{23}}=K_1, \ldots, G_{v_{q_{k-1}}}=K_{q_{k-1}},
G_{v_{(q-1)q}}=K_1, G_{v_{q_k}}=K_{q_k}\}$ and $H={\cal L}(T(\delta(q_1), \ldots, \delta(q_k)))$ is such that $V(H)=P \cup K$, with
$P=\{v_{q_1}, v_{12}, v_{23}, \ldots, v_{(k-1)k}, v_{q_k}\}$, $K=\{v_{q_i}: q_i>0, 2 \le i \le k-1\}$,
$H[P]$ is the path defined by the sequence of $k+1$ vertices $v_{q_1}, v_{12}, v_{23}, \ldots v_{(k-1)k}, v_{q_k}$
and for all $v_{q_i} \in K$, $v_{(i-1)i}v_{q_i}, v_{q_i}v_{i(i+1)} \in E(H)$.
\end{theorem}

\begin{example}\label{ex1}
Consider the caterpillar $T(q_1, q_2, 0, 0, 0, q_6, q_7, q_8, q_9)$, with $q_1>0, q_2>0, q_6>0, q_7>0, q_8>0, q_9>0$.
Then, ${\cal L}\left(T(q_1, q_2, 0, 0, 0, q_6, q_7, q_8, q_9)\right)$ is the $H$-join of the family of graphs
$F=\{G_{q_1}=K_{q_1}, G_{v_{12}}=K_1, G_{q_2}=K_{q_2}, G_{v_{23}}=K_1, G_{v_{34}}=K_1,
G_{v_{45}}=K_1, G_{v_{56}}=K_1, G_{q_6}=K_{q_6}, G_{v_{67}}=K_1, G_{q_7}=K_{q_7}, G_{v_{78}}=K_1, G_{q_8}=K_{q_8},
G_{v_{89}}=K_1, G_{q_9}=K_{q_9}\}$, with
$$
H={\cal L}\left(T(1, 1, 0, 0, 0, 1, 1, 1, 1)\right).
$$
The graph $H$ is depicted in Figure~\ref{figura_1}.
\end{example}

\begin{figure}[ht]
\begin{center}
\unitlength=0.25 mm
\begin{picture}(400,120)(75,60)
%
\put(75,85){\circle*{5.7}}  
\put(100,135){\circle*{5.7}} 
\put(125,85){\circle*{5.7}}  
\put(150,135){\circle*{5.7}} 
\put(200,135){\circle*{5.7}} 
\put(250,135){\circle*{5.7}} 
\put(300,135){\circle*{5.7}} 
\put(325,85){\circle*{5.7}} 
\put(350,135){\circle*{5.7}} 
\put(375,85){\circle*{5.7}} 
\put(400,135){\circle*{5.7}} 
\put(425,85){\circle*{5.7}} 
\put(450,135){\circle*{5.7}} 
\put(475,85){\circle*{5.7}} 
%
\put(75,75){\makebox(0,0){$v_{q_1}$}}
\put(100,145){\makebox(0,0){$v_{12}$}}
\put(125,75){\makebox(0,0){$v_{q_2}$}}
\put(150,145){\makebox(0,0){$v_{23}$}}
\put(200,145){\makebox(0,0){$v_{34}$}}
\put(250,145){\makebox(0,0){$v_{45}$}}
\put(300,145){\makebox(0,0){$v_{56}$}}
\put(325,75){\makebox(0,0){$v_{q_6}$}}
\put(350,145){\makebox(0,0){$v_{67}$}}
\put(375,75){\makebox(0,0){$v_{q_7}$}}
\put(400,145){\makebox(0,0){$v_{78}$}}
\put(425,75){\makebox(0,0){$v_{q_8}$}}
\put(450,145){\makebox(0,0){$v_{89}$}}
\put(475,75){\makebox(0,0){$v_{q_9}$}}
%
\put(75,85){\line(1,2){25}}  
\put(100,135){\line(1,0){50}} 
\put(125,85){\line(1,2){25}}  
\put(125,85){\line(-1,2){25}} 
\put(150,135){\line(1,0){50}} 
\put(200,135){\line(1,0){50}} 
\put(250,135){\line(1,0){50}} 
\put(300,135){\line(1,0){50}} 
\put(325,85){\line(1,2){25}}  
\put(325,85){\line(-1,2){25}} 
\put(350,135){\line(1,0){50}} 
\put(375,85){\line(1,2){25}}  
\put(375,85){\line(-1,2){25}} 
\put(400,135){\line(1,0){50}} 
\put(425,85){\line(1,2){25}}  
\put(425,85){\line(-1,2){25}} 
\put(475,85){\line(-1,2){25}} 
\put(270,50){\makebox(0,0){{\footnotesize The graph $H$ of Example~\ref{ex1}}}}
\end{picture}
\end{center}
\caption{The graph $H$ such that ${\cal L}(T(q_1,q_2,0,0,0,q_6,q_7,q_8,q_9))$, with $q_1>0$, $q_2>0$, $q_6>0$
         $q_7>0$, $q_8>0$ and $q_9>0$, is the $H$-join of the family $\mathcal{F}$ of graphs of Example~\ref{ex1}.}
\label{figura_1}
\end{figure}

Now, it is worth pointing out the following previous result.

\begin{theorem}\label{H-join spectra}\cite{CFMR2013}
Let $\mathcal{F}$ be a family of $d_j$-regular graphs, $G_j$ of order $n_j$, for $1 \le j \le k$
and let $H$ be a graph such that $V(H)=\{1,\dots ,k\}$. Then
\begin{equation*}
\sigma (\bigvee_{H}\mathcal{F})=\left( \bigcup_{j=1}^{k}{\left( \sigma(G_{j})\setminus
                                \{d_{j}\}\right)}\right) \cup
                                \sigma (C),
\end{equation*}
where the symmetric matrix $C$ is as follows:
\begin{equation}
C = \begin{pmatrix}
          d_1         & \rho _{1,2} & \ldots & \rho _{1,k-1} & \rho _{1,k} \\
          \rho _{1,2} & d_2         & \ldots & \rho _{2,k-1} & \rho _{2,k} \\
          \vdots      & \vdots      & \ddots & \vdots        & \vdots \\
          \rho _{1,k} & \rho _{2,k} & \ldots & \rho _{k-1,k} & d_k
    \end{pmatrix}
\end{equation}
with $\rho _{i,j} =\left\{\begin{array}{lll}
                          \sqrt{n_{i}n_{j}} &  & \hbox{if } ij \in E(H) \\
                           0                &  & \hbox{otherwise,}
                          \end{array}%
                   \right.$ for all $i \in \{1,\ldots ,k-1\}$ and $j \in \{i+1,\ldots ,k\}$.
\end{theorem}

\begin{corollary}\label{corolario6}
Consider the caterpillar $T(q_1, \ldots, q_k)$. Then
$$
\sigma({\cal L}(T(q_1, \ldots, q_k))) = \{-1^{[\sum_{r=1}^{k}{q_r-\delta(q_r)}]}\} \cup \sigma^*(C(q_1, \ldots, q_k)),
$$
where $C(q_1, \ldots, q_k)$ is the following $(2k-1)\times(2k-1)$ matrix:
{\footnotesize $$
\bordermatrix{   &v_{q_1}   &v_{12}    &v_{q_2}   &v_{23}    &v_{q_3}    &v_{34}    &v_{q_4}   &\cdots &v_{q_{k-1}}   &v_{(k-1)k}    & v_{q_k}  \cr
      v_{q_1}    &q^+_1-1   &\sqrt{q_1}&   0      &  0       &  0        & 0        &     0    &\cdots &   0          &     0        &  0       \cr
      v_{12}     &\sqrt{q_1}&  0       &\sqrt{q_2}&  1       &  0        & 0        &     0    &\cdots &   0          &     0        &  0       \cr
      v_{q_2}    &  0       &\sqrt{q_2}&q^+_2-1   &\sqrt{q_2}&  0        & 0        &     0    &\cdots &   0          &     0        &  0       \cr
      v_{23}     &  0       &  1       &\sqrt{q_2}&  0       & \sqrt{q_3}& 1        &     0    &\cdots &   0          &     0        &  0       \cr
      v_{q_3}    &  0       &  0       &   0      &\sqrt{q_3}&q^+_3-1    &\sqrt{q_3}&     0    &\cdots &   0          &     0        &  0       \cr
      v_{34}     &  0       &  0       &   0      &  1       & \sqrt{q_3}& 0        &\sqrt{q_4}&\cdots &   0          &     0        &  0       \cr
      v_{q_4}    &  0       &  0       &   0      &  0       &  0        &\sqrt{q_4}&q^+_4-1   &\cdots &   0          &     0        &  0       \cr
      \vdots     &\vdots    &\vdots    &\vdots    &\vdots    &\vdots     &\vdots    &\vdots    &\ddots &\vdots        & \vdots       &\vdots    \cr
      v_{q_{k-1}}&  0       &  0       &   0      &  0       &  0        & 0        &     0    &\cdots &q^+_{k-1}-1   &\sqrt{q_{k-1}}&  0       \cr
      v_{(k-1)k} &  0       &  0       &   0      &  0       &  0        & 0        &     0    &\cdots &\sqrt{q_{k-1}}&     0        &\sqrt{q_k}\cr
      v_{q_k}    &  0       &  0       &   0      &  0       &  0        & 0        &     0    &\cdots &   0          & \sqrt{q_k}   &q^+_k-1   \cr},
$$}
\noindent where $q^+_i-1=\max\{0,q_i-1\}$ for $1 \le i \le k$, and $\sigma^*(C(q_1, \ldots, q_k))$ is the spectrum of the matrix
obtained from $C(q_1, \ldots, q_k)$ after deleting the all zeros rows and columns.
\end{corollary}

\begin{example}\label{ex2}
Consider the graph $H$ such that ${\cal L}(T(q_1,q_2,q_3,q_4))$, with $q_1=4$, $q_2=9$, $q_3=0$ and $q_4=1$,
is the $H$-join of the family of graphs $\mathcal{F}=\{K_4,K_1, K_9, K_1, K_0, K_1, K_1\}$,
where $H={\cal L}(T(\delta(4),\delta(9), \delta(0), \delta(1)))$. Then, according to Corollary~\ref{corolario6}, the
matrix $C(4,9,0,1)$ is as follows:
$$
\bordermatrix{   &v_{q_1} &v_{12}   &v_{q_2} &v_{23}  &v_{q_3}  &v_{34}  &v_{q_4} \cr
      v_{q_1}    &  3     &\sqrt{4} &   0    &  0     &  0      & 0      &     0  \cr
      v_{12}     &\sqrt{4}&  0      &\sqrt{9}&  1     &  0      & 0      &     0  \cr
      v_{q_2}    &  0     &\sqrt{9} &   8    &\sqrt{9}&  0      & 0      &     0  \cr
      v_{23}     &  0       &  1    &\sqrt{9}&  0     & \sqrt{0}& 1      &     0  \cr
      v_{q_3}    &  0       &  0    &   0    &\sqrt{0}&  0      &\sqrt{0}&     0  \cr
      v_{34}     &  0       &  0    &   0    &  1     & \sqrt{0}& 0      &\sqrt{1}\cr
      v_{q_4}    &  0       &  0    &   0    &  0     &  0      &\sqrt{1}&     0  \cr}.
$$
\end{example}

From now on, let us define $p(q_1, \ldots, q_k; \lambda)=\det (C(q_1, \ldots, q_k)-\lambda I)$.
For the particular cases of $k \in \{1, 2, 3\}$, we obtain:
\begin{eqnarray}
p(q_1;\lambda)        &=& q^+_1-1 - \lambda \label{pol_1}.
\end{eqnarray}
\begin{eqnarray}
p(q_1,q_2;\lambda)    &=& \det \left(\begin{array}{ccc}
                                         q^+_1-1 - \lambda &\sqrt{q_1}&   0      \\
                                         \sqrt{q_1}        &-\lambda  &\sqrt{q_2}\\
                                           0               &\sqrt{q_2}&q^+_2-1 - \lambda
                                        \end{array}\right) \nonumber\\
                       &=& (\lambda^2 - (q^+_1-1)\lambda -q_1)(q^+_2-1-\lambda)- (q^+_1-1-\lambda)q_2.\label{pol_2}
\end{eqnarray}

\begin{eqnarray}
p(q_1,q_2,q_3;\lambda)&=& \det \left(\begin{array}{ccccc}
     q_1^+-1-\lambda &\sqrt{q_1}& 0               & 0        & 0\\
     \sqrt{q_1}      &-\lambda  & \sqrt{q_2}      & 1        & 0\\
     0               &\sqrt{q_2}&q^+_2-1 -\lambda &\sqrt{q_2}& 0\\
     0               & 1        & \sqrt{q_2}      &-\lambda  & \sqrt{q_3}\\
     0               & 0        & 0               &\sqrt{q_3}& q_3^+-1-\lambda
     \end{array}\right) \nonumber\\
&=& (\lambda^2 -(q^+_1-1)\lambda - q_1)p(q_2,q_3,\lambda) \nonumber\\
& & + (q^+_1-1-\lambda) \left(q_2(2+\lambda) - (q^+_2-1-\lambda)\right)(q^+_3-1-\lambda) \nonumber\\
& & + (q^+_1-1-\lambda)q_2q_3.\label{pol_3}
\end{eqnarray}

Taking into account Corollary~\ref{corolario6}, we may conclude the following result.

\begin{lemma}\label{recorrencia}
Consider the caterpillar $T(q_1, \ldots, q_k)$, with $k \ge 4$. Then $p(q_1,\ldots,q_k;\lambda)=$
\begin{eqnarray*}
&=& \left(\lambda^2-(q^+_1-1)\lambda - q_1\right)p(q_2, \ldots, q_k;\lambda)\\
& & + (q_1^+-1-\lambda)\sum_{j=2}^{k-1}{(-1)^{j}\left(\prod_{i=2}^{j-1}{q_i}\right)\left(q_{j}(2+\lambda)-(q_{j}^+-1-\lambda)\right)p(q_{j+1}, \ldots,q_k;\lambda)}\\
& & + (-1)^{k+1}(q_1^+-1-\lambda)q_2 \cdots q_{k-2}q_{k-1}q_k,
\end{eqnarray*}
assuming that $\prod_{2}^{1}{q_i}=1$.
\end{lemma}

\begin{proof}
Consider the matrix $C=C(q_1, \ldots, q_k)$ of Corollary~\ref{corolario6} and apply the generalized Laplacian theorem
to the first two rows of the matrix $C(q_1, \ldots, q_k)-\lambda I.$ Thus we obtain, $p(q_1, \ldots, q_k;\lambda)=$
{\footnotesize
\begin{eqnarray}
& & \det \left(\begin{array}{cc}
          q^+_1-1 -\lambda &\sqrt{q_1} \\
         \sqrt{q_1}        &  -\lambda
               \end{array}\right) \nonumber  \\
& & \det \left(\begin{array}{ccccccc}\label{padrao1}
          q^+_2-1 - \lambda &\sqrt{q_2}&  0              & 0        &\cdots &     0        &  0       \\
          \sqrt{q_2}        &-\lambda  & \sqrt{q_3}      & 1        &\cdots &     0        &  0       \\
              0             &\sqrt{q_3}&q^+_3-1 -\lambda &\sqrt{q_3}&\cdots &     0        &  0       \\
              0             &  1       & \sqrt{q_3}      & -\lambda &\cdots &     0        &  0       \\
              0             &  0       &  0              &\sqrt{q_4}&\cdots &     0        &  0       \\
          \vdots            &\vdots    &\vdots           &\vdots    &\ddots & \vdots       &\vdots    \\
              0             &  0       &  0              & 0        &\cdots &  -\lambda    &\sqrt{q_k}\\
              0             &  0       &  0              & 0        &\cdots & \sqrt{q_k}   &q^+_k-1 -\lambda
              \end{array}\right)\\
&-& \det \left(\begin{array}{cc}
            q^+_1-1 - \lambda & 0 \\
           \sqrt{q_1}         & \sqrt{q_2}
                \end{array}\right) \nonumber \\
& & \det \left(\begin{array}{ccccccccc}\label{padrao2}
          \sqrt{q_2} &\sqrt{q_2} & 0              & 0         & 0               &\cdots &     0        &  0       \\
            1        &-\lambda   &\sqrt{q_3}      & 1         & 0               &\cdots &     0        &  0       \\
            0        &\sqrt{q_3} &q^+_3-1-\lambda & \sqrt{q_3}& 0               &\cdots &     0        &  0       \\
            0        &  1        &\sqrt{q_3}      & -\lambda  & \sqrt{q_4}      &\cdots &     0        &  0       \\
            0        &  0        & 0              & \sqrt{q_4}& q^+_4-1-\lambda &\cdots &     0        &  0       \\
         \vdots      &\vdots     &\vdots          &\vdots     & \vdots          &\ddots & \vdots       &\vdots    \\
            0        &  0        & 0              & 0         & 0               &\cdots &  -\lambda    &\sqrt{q_k}\\
            0        &  0        & 0              & 0         & 0               &\cdots & \sqrt{q_k}   &q^+_k-1 -\lambda
        \end{array}\right)\\
&+& \det \left(\begin{array}{cc}
          q^+_1-1 - \lambda & 0 \\
         \sqrt{q_1}         & 1
               \end{array}\right)\nonumber\\
& & \det \left(\begin{array}{ccccccc}\label{padrao3}
          \sqrt{q_2}  &q_2^+-1-\lambda &  0                & 0         &\cdots &     0        &  0       \\
            1          &\sqrt{q_2}     & \sqrt{q_3}        & 1         &\cdots &     0        &  0       \\
            0          &  0            & q^+_3-1 - \lambda & \sqrt{q_3}&\cdots &     0        &  0       \\
            0          &  0            & \sqrt{q_3}        & -\lambda  &\cdots &     0        &  0       \\
            0          &  0            &  0                &\sqrt{q_4} &\cdots &     0        &  0       \\
           \vdots      &\vdots         &\vdots             &\vdots     &\ddots &\vdots        &\vdots    \\
            0          &  0            &  0                & 0         &\cdots &-\lambda      &\sqrt{q_k}\\
            0          &  0            &  0                & 0         &\cdots &\sqrt{q_k}    &q^+_k-1 -\lambda
        \end{array}\right)
\end{eqnarray}}

Let us denote the determinant of type \eqref{padrao2} by $g(q_i,\ldots,q_k;\lambda)$, for $2 \le i \le k-1$, that is,
\begin{eqnarray*}
g(q_i,\ldots,q_k;\lambda)&=& \det \left(\begin{array}{ccccccccc}
                             \sqrt{q_i} &\sqrt{q_i} & 0              & 0         & 0               &\cdots &     0        &  0       \\
                               1        &-\lambda   &\sqrt{q_{i+1}}  & 1         & 0               &\cdots &     0        &  0       \\
                               0        &\sqrt{q_{i+1}} &q^+_{i+1}-1-\lambda & \sqrt{q_{i+1}}& 0               &\cdots &     0        &  0       \\
                               0        &  1        &\sqrt{q_{i+1}}      & -\lambda  & \sqrt{q_{i+2}}      &\cdots &     0        &  0       \\
                               0        &  0        & 0              & \sqrt{q_{i+2}}& q^+_{i+2}-1-\lambda &\cdots &     0        &  0       \\
                              \vdots    &\vdots     &\vdots          &\vdots     & \vdots          &\ddots & \vdots       &\vdots    \\
                               0        &  0        & 0              & 0         & 0               &\cdots &  -\lambda    &\sqrt{q_k}\\
                               0        &  0        & 0              & 0         & 0               &\cdots & \sqrt{q_k}   &q^+_k-1 -\lambda
                                     \end{array}\right)\\
                       & & \text{ for } \qquad 2 \le i \le k-1,
\end{eqnarray*}
and the determinant of type \eqref{padrao3} by $h(q_i, \ldots,q_k;\lambda)$, for $2 \le i \le k-1$, that is,
\begin{eqnarray*}
h(q_i,\ldots,q_k;\lambda)&=& \det \left(\begin{array}{ccccccc}
          \sqrt{q_i}  &q_i^+-1-\lambda &  0                    & 0             &\cdots &     0        &  0       \\
            1         &\sqrt{q_i}      & \sqrt{q_{i+1}}        & 1             &\cdots &     0        &  0       \\
            0         &  0             & q^+_{i+1}-1 - \lambda & \sqrt{q_{i+1}}&\cdots &     0        &  0       \\
            0         &  0             & \sqrt{q_{i+1}}        & -\lambda      &\cdots &     0        &  0       \\
            0         &  0             &  0                    &\sqrt{q_{i+2}} &\cdots &     0        &  0       \\
           \vdots     &\vdots          &\vdots                 &\vdots         &\ddots &\vdots        &\vdots    \\
            0         &  0             &  0                    & 0             &\cdots &-\lambda      &\sqrt{q_k}\\
            0         &  0             &  0                    & 0             &\cdots &\sqrt{q_k}    &q^+_k-1 -\lambda
        \end{array}\right)\\
                       & & \text{ for } \qquad 2 \le i \le k-1.
\end{eqnarray*}
Then
\begin{eqnarray*}
h(q_i,\ldots,q_k;\lambda) &=& \det \left(\begin{array}{cc}
                                         \sqrt{q_i} & q_i^+-1-\lambda \\
                                          1         & \sqrt{q_i}
                                         \end{array}\right)p(q_{i+1}, \ldots, q_k;\lambda)\\
                          &=& (q_i - (q_i^+-1-\lambda))p(q_{i+1}, \ldots, q_k;\lambda)
\end{eqnarray*}
and
\begin{eqnarray*}
g(q_i,\ldots,q_k;\lambda) &=& \det \left(\begin{array}{cc}
                                             \sqrt{q_i} & \sqrt{q_i} \\
                                              1         & -\lambda
                                                   \end{array}\right)p(q_{i+1}, \ldots, q_k; \lambda)\\
                          &-& \det \left(\begin{array}{cc}
                                              \sqrt{q_i} & 0 \\
                                               1         & \sqrt{q_{i+1}}
                                                    \end{array}\right)g(q_{i+1}, \ldots, q_k;\lambda)\\
                          &+& \det \left(\begin{array}{cc}
                                              \sqrt{q_i} & 0 \\
                                                    1    & 1
                                                   \end{array}\right)h(q_{i+1},\ldots,q_k;\lambda)\\
                          &=& -\sqrt{q_i}(1+\lambda)p(q_{i+1}, \ldots, q_k, \lambda)\\
                          & & - \sqrt{q_{i}}\sqrt{q_{i+1}}g(q_{i+1}, \ldots, q_k;\lambda)
                              + \sqrt{q_i}(q_{i+1} - (q_{i+1}^+-1-\lambda))p(q_{i+2}, \ldots, q_k;\lambda).
\end{eqnarray*}

Since the determinant in \eqref{padrao1} is equal to $p(q_2, \ldots, q_k;\lambda)$, it follows that
\begin{eqnarray*}
p(q_1, \ldots, q_k;\lambda) &=& \left(\lambda^2-(q^+_1-1)\lambda - q_1\right) p(q_2, \ldots, q_k;\lambda)\\
                            & & + (q_1^+-1-\lambda)h(q_2, \ldots,q_k;\lambda)\\
                            & & - \sqrt{q_2}(q_1^+-1-\lambda)g(q_2,\ldots,q_k;\lambda).
\end{eqnarray*}
Expanding $g(q_2,\ldots, q_k;\lambda)$, we obtain
\begin{eqnarray*}
p(q_1, \ldots, q_k;\lambda)&=& \left(\lambda^2-(q^+_1-1)\lambda - q_1\right)p(q_2, \ldots, q_k;\lambda)\\
& & + (q_1^+-1-\lambda)\left(q_2(2+\lambda)-(q_2^+-1-\lambda)\right)p(q_3, \ldots,q_k;\lambda)\\
& & - (q_1^+-1-\lambda)q_2\left(q_3(2+\lambda)-(q_3^+-1-\lambda)\right)p(q_4,\ldots,q_k;\lambda)\\
& & + (q_1^+-1-\lambda)q_2q_3\left(q_4(2+\lambda)-(q_4^+-1-\lambda)\right)p(q_5,\ldots,q_k;\lambda)\\
& & \vdots \\
& & + (-1)^{k-2}(q_1^+-1-\lambda)q_2 \cdots q_{k-3}\left(q_{k-2}(2+\lambda)-(q_{k-2}^+-1-\lambda)\right)p(q_{k-1},q_k;\lambda)\\
& & + (-1)^{k-1}(q_1^+-1-\lambda)q_2q_3\cdots q_{k-2}(q_{k-1}-(q_{k-1}^+-1-\lambda))p(q_k;\lambda)\\
& & + (-1)^{k-1}(q_1^+-1-\lambda)q_2 \cdots q_{k-2}q_{k-1}(q_k^+-1-\lambda)(1+\lambda)\\
& & + (-1)^{k-1}(q_1^+-1-\lambda)q_2 \cdots q_{k-2}q_{k-1}q_k.\\
\end{eqnarray*}
Thus the result is derived.
\end{proof}

As a consequence of Corollary~\ref{corolario6} and Theorem~\ref{basic_result}, applying Lemma~\ref{recorrencia}, we are able to determine
the characteristic polynomial of the Laplacian matrix of a caterpillar $T(q_1,\ldots,q_k)$.

\begin{theorem}\label{laplacian_caterpillar_spectra}
Consider a caterpillar $T(q_1, \ldots, q_k)$, with $k \ge 4$. Then the Laplacian eigenvalues of $T(q_1, \ldots, q_k)$ are roots
of the polynomial
$$
\det\left(L(T)-\mu I)\right) = \mu \frac{(\mu-1)^{\sum_{r=1}^{k}{(q_r-\delta(q_r))}}}{(\mu-2)^{(k-\sum_{r=1}^{k}{\delta(q_r)})}}p(q_1,\ldots,q_k;\mu-2).
$$
\end{theorem}

\begin{proof}
Let $T=T(q_1, \ldots, q_k)$. By Theorem~\ref{basic_result},
\begin{equation}\label{laplacian_spectra_1}
\det\left(L(T)-\mu I)\right)=\mu \det\left(A_{\cal L}(T)-(\mu-2)I)\right).
\end{equation}
On the other hand, by Corollary~\ref{corolario6}, the eigenvalues of $A_{\cal L}(T)$ are $-1$ with multiplicity
$q=\sum_{r=1}^{k}{(q_r-\delta(q_r))}$ and the eigenvalues of the matrix $C=C(q_1, \ldots, q_r)$,
after deleting its all zero rows and columns. That is,
\begin{equation}\label{laplacian_spectra_2}
\det\left(A_{\cal L}(T)-\lambda I)\right) = \frac{(\lambda+1)^q}{\lambda^{k-\sum_{r=1}^{k}\delta(q_r)}} \det\left(C - \lambda I\right).
\end{equation}
Notice that the power of $\lambda$ in the denominator corresponds to the number of all zero rows (columns)
of the matrix $C$ that should be eliminated. Then, since $\det\left(C - (\mu-2) I\right) =  p(q_1,\ldots,q_k;\mu-2)$,
taking into account \eqref{laplacian_spectra_1}, the result follows.
\end{proof}

\section{Lower and upper bounds on the algebraic connectivity of caterpillars}
Since, when $G$ is bipartite, the nonzero eigenvalues of the Laplacian matrix of $G$ are the eigenvalues of its line graph, ${\cal L}(G),$
plus $2$, and when the graph is a tree, $T$, with at least two edges (that is, $T$ is not complete), the least eigenvalue of ${\cal L}(T)$
is greater than $-2$ and less than or equal $-1$, it follows that the algebraic connectivity of $T$ is greater than $0$ and less than or
equal $1$. Therefore, if $T$ is a caterpillar, $T(q_1, \ldots, q_k)$, taking into account \eqref{laplacian_spectra_1} and \eqref{laplacian_spectra_2},
its algebraic connectivity is the least eigenvalue of the matrix  $C(q_1, \ldots, q_k)$ plus $2$, that is, the minimal root of the polynomial
$p(q_1, \ldots, q_k;\lambda-2)$. From now on, the algebraic connectivity of the caterpillar $T(q_1, \ldots, q_k)$ is
denoted $\mu(q_1, \ldots, q_k)$.\\

In the next subsections, we use two approaches in order to obtain bounds on the algebraic connectivity of caterpillars.
In the first approach, taking advantage of the structure of the matrix $C(q_1, \ldots, q_k),$ we determine an upper
bound on the algebraic connectivity of the caterpillar $T(q_1, \ldots, q_k),$ from the spectrum of caterpillars $T(q_j,q_{j+1}), 1 \le j \le k-1$.
In this approach, the spectrum of these caterpillars are presented in an explicit form and the upper bound on $\mu(q_1, \ldots, q_k)$
is produced with low computational effort. The second approach is based on a result published in \cite{Medina} and, despite to need more
computational effort than the previous approach, accordingly with the computational experiments presented in the end of the paper,
the produced lower and upper bounds are very accurate.

\subsection{Upper bounds on the algebraic connectivity of caterpillars determined using Cardano's formula}
Considering a caterpillar $T(q_1, \ldots, q_k)$, with $k \ge 4$ and $q_1 \ne 0 \ne q_k$, and deleting the fourth column and the
fourth row of the matrix $C=C(q_1, \ldots, q_k)$, we obtain a submatrix $C'$ with two diagonal blocks, $C(q_1,q_2)$ and $C(q_3,\ldots,q_k)$.
The spectrum of $C'$ is the union of the spectrum of $C(q_1,q_2)$ with the spectrum of $C(q_3,\ldots,q_k)$. By interlacing,
it follows that $\lambda_{2k-1}(C) \le \lambda_{2k-2}(C') \le \cdots \le \lambda_{1}(C') \le \lambda_1(C)$ and then
$\lambda_{2k-1}(C) \le \lambda_3(C(q_1,q_2)).$

\begin{theorem}
Let  $T(q_1, \ldots, q_k)$ be a caterpillar, with $k \ge 4$. Then the least eigenvalue of the matrix $C=C(q_1, \ldots, q_k)$, $\lambda_{2k-1}(C)$,
has the upper bound:
$$
\lambda_{2k-1}(C) \le \min_{j \in \{1, \ldots, k-1\}} \lambda_3(C(q_j,q_{j+1})),
$$
where $\lambda_3(C(q_j,q_{j+1}))$ is the least eigenvalue of $C(q_j,q_{j+1})$.
\end{theorem}

\begin{proof}
If $j=1$ the result follows from the above analysis. If $j>1$, we may consider the line graph of the graph formed by a pair of caterpillars
obtained by deleting the edge $v_{q_{j-1}}v_{q_j}$. The spectrum of this line graph includes the union of the spectrum of the matrices
$C(q_1, \ldots, q_{j-1})$ and $C(q_j, \ldots, q_k)$. If $k =j+1$, then $C(q_j, \ldots, q_k)=C(q_j,q_{j+1})$ and the result holds.
If $k>j+1$, the above procedure must be repeated to the matrix $C(q_j, \ldots, q_k)$, and the obtained matrix
$C'$ has the diagonal blocks $C(q_j,q_{j+1})$ and $C(q_{j+2}, \ldots, q_{k})$. In any case, we have
$$
\lambda_{2k-1}(C) \le \lambda_{2(k-j)+1}(C(q_j, \ldots, q_k)) \le \lambda_3(C(q_j,q_{j+1})).
$$
\end{proof}

\begin{corollary}
If $T(q_1, \ldots, q_k)$ is a caterpillar, with $k \ge 4$ and $q_1 \ne 0 \ne q_k$, then
\begin{eqnarray}
\mu(q_1, \ldots, q_k) & \le & \min_{j \in \{1, \ldots, k-1\}} \lambda_3(C(q_j,q_{j+1}))+2.\label{ub_1}
\end{eqnarray}
\end{corollary}

The next theorem characterizes the three eigenvalues of a matrix $C(q_1,q_2)$.

\begin{theorem}
Let $T(q_1,q_2)$ be a caterpillar such that $q_1+q_2>0$ and $C=C(q_1,q_2)$.
\begin{enumerate}
\item If $q_2=0$, then $\lambda_3(C) = \lambda_2(C) = -1$ and $\lambda_1(C)=q_1$.
\item If $q_1 \ne 0 \ne q_2$, then $\lambda_1(C), \lambda_2(C), \lambda_3(C) \in \{\zeta_1, \zeta_2, \zeta_3\}$, with $\zeta_j, j=0, 1, 2,$
      determined as follows:
      \begin{eqnarray*}
      \zeta_j &=& 2\sqrt{-\frac{r}{3}}\cos{(\frac{\theta+2\pi j}{3})} + \frac{q_1 + q_2 - 2}{3},
      \end{eqnarray*}
      where $\theta = \arg{\left(- s/2 + i \sqrt{-(\frac{r}{3})^3-(\frac{s}{2})^2}\right)}$ and
      \begin{eqnarray*}
      r &=& ((q_1 - 1)(q_2 - 1) - q_1 - q_2) - \frac{\left(2 - q_1 - q_2\right)^2}{3}, \\
      s &=& 2\left(\frac{2 - q_1 - q_2}{3}\right)^3 - \frac{\left(2 - q_1 - q_2\right)\left((q_1 - 1)(q_2 - 1) - q_1 - q_2\right)}{3} \\
        & & + q_1(q_2 -1) + q_2(q_1 -1).
      \end{eqnarray*}
\end{enumerate}
\end{theorem}

\begin{proof}
Let us prove each part of this theorem independently.
\begin{enumerate}
\item Assuming $q_1 \ne 0 = q_2$, it follows that ${\cal L}(T(q_1,0))$ is a complete graph of order $q_1+1$ and then its adjacency eigenvalues
      are $-1$ and $q_1$. Therefore the matrix $C(q_1,0)$ has the eigenvalues $-1$ with multiplicity $2,$ and $q_1$ with multiplicity $1$.
\item Consider that  $q_1 \ne 0 \ne q_2$. From \eqref{pol_2}, since $q_1^+=q_1$ and $q_2^+=q_2$, we obtain the monic polynomial,
      \begin{equation*}
      -p(q_1,q_2; \zeta) = \zeta^3 + (2 - q_1 - q_2)\zeta^2 + ((q_1 - 1)(q_2 - 1) - q_1 - q_2)\zeta + q_1(q_2 -1) + q_2(q_1 -1),
      \end{equation*}
      which obviously has the same roots as $p(q_1,q_2,\zeta)$. As it is well known, setting $\zeta = \gamma - \frac{2 - q_1 - q_2}{3}$,
      we obtain the depressed form of the cubic equation $-p(q_1,q_2; \zeta)=0$,
      \begin{equation}\label{depressed_cubic_equation}
      \gamma^3 + r\gamma + s = 0,
      \end{equation}
      where $r$ and $s$ are defined as stated in the theorem. Now, applying the Cardano's formula for the depressed
      cubic equation \eqref{depressed_cubic_equation}, it follows that
      \begin{equation}\label{least_eigenvalue}
      \gamma = \left(- s/2 + \sqrt{(\frac{r}{3})^3+(\frac{s}{2})^2}\right)^{1/3} + \left(- s/2 - \sqrt{(\frac{r}{3})^3+(\frac{s}{2})^2}\right)^{1/3}.
      \end{equation}
      The discriminant of the depressed cubic equation \eqref{depressed_cubic_equation} is $\Delta = -4 r^3 - 27 s^2$
      and the cubic equation \eqref{depressed_cubic_equation} has three real roots if and only if $\Delta \ge 0$
      (see \cite{Vinberg2003}).  In this case, since the roots of $p(q_1,q_2; \lambda)$ are real, the roots of \eqref{depressed_cubic_equation}
      are also real.

      However, since $(\frac{r}{3})^3 + (\frac{s}{2})^2 = \frac{-\Delta}{108} < 0,$ \eqref{least_eigenvalue} must be determined in $\mathbb{C}$
      (recall that when we add complex conjugates,  we get double their real part) and then
      $\gamma = \left(- s/2 + i \sqrt{-(\frac{r}{3})^3-(\frac{s}{2})^2}\right)^{1/3}+\left(- s/2 - i \sqrt{-(\frac{r}{3})^3-(\frac{s}{2})^2}\right)^{1/3}.$\\
      Setting $\theta = \arg{\left(- s/2 + i \sqrt{-(\frac{r}{3})^3-(\frac{s}{2})^2}\right)},$ the three roots are:
      \begin{eqnarray*}
      \gamma_j &=& \left(\sqrt{s^2/4 + (-(\frac{r}{3})^3 - (\frac{s}{2})^2)}\right)^{1/3}\left(\exp^{\frac{\theta + 2\pi j}{3}i}+\exp^{-\frac{ \theta + 2\pi j}{3}i}\right)\\
               &=& 2\sqrt{-\frac{r}{3}} \cos{(\frac{\theta + 2\pi j}{3})}, \text{ for } j=0, 1, 2.
      \end{eqnarray*}
      Therefore, since the roots of $p(q_1,q_2; \zeta)$ are $\zeta_j = \gamma_j - \frac{2-q_1-q_2}{3}, j=0, 1, 2$,
      \begin{eqnarray*}
      \zeta_j &=& 2\sqrt{-\frac{r}{3}}\cos{(\frac{\theta + 2\pi j}{3})} + \frac{q_1+q_2-2}{3}, \text{ for } j=0, 1, 2.
      \end{eqnarray*}
\end{enumerate}
\end{proof}

It should be noted that $\lambda_3(C(q_1,q_2)) = \min_{j \in \{0,1,2\}}{\zeta_j}$.

\begin{example}
Let us determine an upper bound \eqref{ub_1} for the algebraic connectivity of the caterpillar
$T\left( 4,9,0,1\right)$ of Example~\ref{ex2} (which is $\mu(4,9,0,1))=0.1862$).\\
Since $\lambda_3(C(4,9))=-1.3955$ and $\lambda_3(C(9,0))=\lambda_3(C(0,1))=-1,$ then
$$
\mu(4,9,0,1) \le -1.3955+2=0.6045.
$$
\end{example}

\subsection{Accurate upper and lower bounds on the algebraic connectivity of caterpillars}

First, we recall the following result published in \cite{Medina}, where $\text{tr}(M)$ denotes the trace of the
square matrix $M$ (this notation is used throughout this section).

\begin{theorem}\cite{Medina}\label{alg_upper_bound}
Let $A$ be an $n\times n$ symmetric matrix with only positive eigenvalues
$\lambda_1 \left( A\right) \ge \lambda _{2}\left(A\right) \ge \cdots \ge \lambda_n\left(A\right)\ge 0$
and $B$ be an $\left( n-1\right) \times \left( n-1\right) $ principal submatrix of $A$ whose eigenvalues
$\lambda_1\left( B\right) \ge \lambda_2\left(B\right) \ge \cdots \ge \lambda_{n-1}\left( B\right) \ge 0,$
interlace the eigenvalues of $A$.
Then $\lambda _{n}\left( A\right) \leq \frac{1}{\text{tr}\left( A^{-1}\right)-\text{tr}\left( B^{-1}\right) }.$
\end{theorem}

Let $X\left( \lambda \right) $ be an $n\times n$ matrix whose entries are real functions of $\lambda $. Then%
\begin{equation}
\frac{d}{d\lambda }\det X\left( \lambda \right)=\sum_{i,j=1}^{n}\left( -1\right) ^{i+j}\det X\left(\begin{tabular}{c|c}
                                                                                                    $i$ & $j$
                                                                                                   \end{tabular}\right)
                                                \frac{d}{d\lambda }x_{ij}\left( \lambda \right),\label{det_derivative}
\end{equation}
where $X\left(\begin{tabular}{c|c}
                     $i$ & $j$%
              \end{tabular}\right) $ denotes the $\left( n-1\right) \times \left( n-1\right) $ matrix
obtained of $X$ deleting the $i$-th row and the $j$-th column of $X\left(\lambda \right)$
(see \cite{HMinc88}).

\begin{corollary}\cite{HMinc88}\label{LemmaMinc}
Let $A$ be an $n \times n$ nonsingular matrix and $\lambda \notin \sigma(A)$. Then
\begin{eqnarray}
\frac{d}{d\lambda }\det \left(A-\lambda I\right)
             &=&-\sum_{i=1}^{n}\det \left( \left( A-\lambda I\right) \left(\begin{tabular}{c|c}
                                                                                  $i$ & $i$%
                                                                           \end{tabular}\right)\right) \nonumber \\
             &=&-\text{tr}(Adj\left( A-\lambda I\right) ) \nonumber \\
             &=&-\det\left( A-\lambda I\right)\text{tr}\left(\left( A-\lambda I\right)^{-1}\right),\label{derivative}
\end{eqnarray}
where $Adj\left( B\right) $ denotes the adjugate or classical adjoint of the square matrix $B$ (that is, the
transpose of the cofactor matrix).
\end{corollary}

As immediate consequence, if $\lambda \notin \sigma(A)$ and $\pi \left( \lambda \right) =\det \left( A-\lambda I\right),$
from \eqref{derivative}, we may conclude that
\begin{equation*}
\text{tr}\left(\left( A-\lambda I\right)^{-1}\right) = -\frac{\frac{d}{d\lambda }\det \left(A-\lambda I\right)}{\det\left( A-\lambda I\right)}
                                                     = -\frac{\pi ^{\prime }\left(\lambda\right) }{\pi \left(\lambda\right)}.
\end{equation*}%
In particular, if $0 \notin \sigma \left( A\right),$ we find
\begin{equation}
\text{tr}\left( A^{-1}\right) =-\frac{\pi ^{\prime }\left( 0\right) }{\pi \left(0\right) }.  \label{inv_trace}
\end{equation}

For $1\leq i \le k-1$, let us denote by $\widetilde{C}_{(i)}(q_1,\ldots,q_k)$ the submatrix obtained from $C(q_1,\ldots ,q_k)$, deleting
its $2i$-th row and column, respectively. If $b=2k-1$ then $\widetilde{C}_{(i)}(q_1,\ldots ,q_k)$ is a $b-1\times b-1$ symmetric matrix,
which is the direct sum of the matrices $C(q_1,\ldots ,q_i)$ and $C(q_{i+1},\ldots ,q_{k})$.


\begin{theorem}\label{lb_and_ub}
Consider a caterpillar $T(q_1,\ldots,q_k)$. Let $C=C(q_1, \ldots, q_k),$ $\mu=\mu (q_1,\ldots ,q_k)$ and
$\widetilde{C}_{(i)}=\widetilde{C}_{(i)}(q_1,\ldots ,q_k)$ as defined above. Assuming that $C$ has order $b$,
then
\begin{equation*}
\frac{1}{\text{tr}((2I_{b}+C)^{-1})} \le \mu \le \min_{1 \le i\le k-1}\frac{1}{\text{tr}((2I_{b}+C)^{-1})-\text{tr}((2I_{b-1}+\widetilde{C}_{(i)})^{-1})}
\end{equation*}
\end{theorem}

\begin{proof}
We will prove each inequality, starting with the proof of the right one. Before that, it should be noted that the least eigenvalue
of the matrix $C$ is greater  than $-2$ and, therefore, the eigenvalues of $2I_b+C$ and $2I_{b-1}+\widetilde{C}_{(i)}$ are all
positive.

\begin{enumerate}
\item The matrix $2I_{b}+C(q_{1},\ldots ,q_{k})$ is permutational similar to a matrix where $2I_{b-1}+\widetilde{C}_{(i)}(q_{1},\ldots ,q_{k})$
      is its principal submatrix. Therefore, by applying Theorem~\ref{alg_upper_bound}, the upper bound is obtained.
\item Now, we use \eqref{inv_trace} to obtain the lower bound on $\mu$. Since the eigenvalues $\beta_1 \ge \beta_2 \ge \cdots \ge \beta_b$
      of $(2I_b+C)^{-1}$ are all positive, the inequality $\beta_1 \le \text{tr}((2I_{b}+C)^{-1})$ holds. Therefore, taking into account
      that $\beta_1 = \frac{1}{\mu},$ it follows that $\frac{1}{\text{tr}((2I_{b}+C)^{-1})}\le \mu.$
\end{enumerate}
\end{proof}

Recalling that $p(q_1,\ldots,q_k;\lambda)=\det (C(q_1,\ldots,q_k)-\lambda I)$, and defining
\begin{equation*}
\pi (q_1,\ldots ,q_k;\lambda )=\det \left( (2I_{b}+C(q_1,\ldots,q_k)\right) - \lambda I_{b}),
\end{equation*}%
it follows that $\pi (q_1,\ldots ,q_k;0)=p(q_1,\ldots,q_k;-2)$ and $\pi^{\prime }(q_1,\ldots,q_k;0)=p^{\prime }(q_1,\ldots,q_k;-2).$

Now, taking into account that from \eqref{inv_trace} we obtain
\begin{equation}
\text{tr}\left((2I_{b}+C(q_1,\ldots ,q_k))^{-1}\right)=-\frac{\pi ^{\prime }(q_1,\ldots,q_k;0)}{\pi (q_1,\ldots ,q_k;0)}
                                                          =-\frac{p^{\prime }(q_1,\ldots,q_k;-2)}{p(q_1,\ldots ,q_k;-2)},\label{identity}
\end{equation}
using Lemma~\ref{recorrencia} we may introduce the following algorithm for the determination of
$\text{tr}\left((2I_{b}+C(q_{1},\ldots ,q_{k}))^{-1}\right).$

\subparagraph{Algorithm for the determination of $\text{tr}\left((2I_{b}+C(q_{1},\ldots ,q_{k}))^{-1}\right)$}
\begin{enumerate}
\item Compute $q_{1}^{+},q_{2}^{+},\ldots ,q_{k}^{+}.$
\item Determine, sequentially, the denominators of \eqref{identity}:
      {\small
      \begin{eqnarray*}
      p(q_1;-2)              & =     & q_1^+ + 1 \\
      p(q_1,q_2;-2)          & =     & (2+2q_1^{+}-q_1)(q_2^{+}+1)-q_2(q_1^{+}+1) \\
      p(q_1,q_2,q_3;-2)      & =     & (2+2q_1^{+}-q_1)p(q_2,q_3;-2)-(q_1^{+}+1)(q_2^{+}+1)(q_3^{+}+1)+\\
                             &       & q_2q_3(q_1^{+}+1)\\
      \vdots                 &\vdots & \vdots \\
      p(q_1,\ldots,q_{k};-2) & =     & (2+2q_{1}^{+}-q_{1})p(q_{2},\ldots,q_{k};-2)+ \\
                             &       & \sum_{j=2}^{k-1}{(-1)^{j}\left( \prod_{i=2}^{j-1}{q_{i}}\right)
                                          (q_{1}^{+}+1)(q_{j}^{+}+1)p(q_{j+1},\ldots ,q_{k};-2)+} \\
                             &       & (-1)^{k+1}(q_{1}^{+}+1)q_{2}\cdots q_{k-2}q_{k-1}q_{k}. \\
      \end{eqnarray*}}
\item Determine, sequentially, the numerators of \eqref{identity}:
      {\small
      \begin{eqnarray*}
      p^{\prime }(q_{a};-2)             & =     & -1 \\
      p^{\prime }(q_{a},q_{b};-2)       & =     & \left( -3-q_{a}^{+}\right)(q_{b}^{+}+1)-(2+2q_{a}^{+}-q_{a})+q_{b}\\
      p^{\prime }(q_{a},q_{b},q_{c};-2) & =     & (-3-q_{a}^{+})p(q_{b},q_{c};-2)+(2+2q_{a}^{+}-q_{a})p^{\prime}(q_{b},q_{c};-2)+\\
                                        &       & (q_{b}^{+}+1)(q_{c}^{+}+1)+\left( q_{b}+1\right)(q_{a}^{+}+1)(q_{c}^{+}+1)+\\
                                        &       & (q_{a}^{+}+1)(q_{b}^{+}+1)-q_{b}q_{c} \\
      \vdots                            &\vdots & \vdots \\
      p^{\prime}(q_{1},\ldots ,q_{k};-2)& =     & \left( -q_{1}^{+}-3)\right)p(q_{2},\ldots ,q_{k};-2)+\left( 2+2q_{1}^{+}-q_{1}\right) p^{\prime}(q_{2},\ldots ,q_{k};-2)+\\
                                        &       & \sum_{j=2}^{k-1}{(-1)^{j}\left(\prod_{i=2}^{j-1}{q_{i}}\right)(\left( q_{j}^{+}+1\right) +}\left( q_{j}+1\right) (q_{1}^{+}+1){)p(q_{j+1},\ldots ,q_{k};-2)+} \\
                                        &       & \sum_{j=2}^{k-1}{(-1)^{j+1}\left(\prod_{i=2}^{j-1}{q_i}\right)(q_1^{+}+1)(q_j^{+}+1)p^{\prime}(q_{j+1},\ldots,q_k;-2)}+ \\
                                        &       & (-1)^{k}q_2\cdots q_{k-2}q_{k-1}q_k. \\
      \end{eqnarray*}}
\item Determine the trace of the matrix $\left(2I_{b}+C(q_{1},\ldots ,q_{k})\right)^{-1},$ setting
      \begin{equation}
      \text{tr}\left((2I_{b}+C(q_{1},\ldots ,q_{k}))^{-1}\right) = -\frac{p^{\prime }(q_{1},\ldots ,q_{k};-2)}{p(q_{1},\ldots ,q_{k};-2)}. \label{inv_lb}
      \end{equation}
\item \textbf{End of the algorithm}.
\end{enumerate}

\begin{example}\label{ex_lb_robbiano}
Let us determine a lower bound on the algebraic connectivity of the caterpillar $T\left( 4,9,0,1\right)$
of Example~\ref{ex2} (which is $\mu(4,9,0,1)=0.1862$) applying Theorem~\ref{lb_and_ub} and using the above algorithm
(notice that $4^{+}=4,$ $9^{+}=9,$ $0^{+}=1$ and $1^{+}=1$). Since
$$
p(1;-2)=2; \;\; p(0,1;-2)=6; \;\; p(9,0,1;-2)=6; \;\; p(4,9,0,1;-2)= 36,
$$
and
$$
p^{\prime }(1;-2) =-1; \; p^{\prime }(9,0,1;-2)=-149; \; p^{\prime }(4,9,0,1;-2)=-382,
$$
it follows that $\frac{1}{\text{tr}\left((2I_{b}+C(q_{1},\ldots ,q_{k}))^{-1}\right)} = \left( \frac{382}{36}\right)^{-1} = 0.0942 \le \mu =0.1862.$
\end{example}

We recall that $\widetilde{C}_{(i)}(q_1,\ldots ,q_k)$ is a $b-1\times b-1$ symmetric matrix, which is the direct sum of the matrices
$C(q_1,\ldots,q_i)$ and $C(q_{i+1},\ldots ,q_k)$. As a consequence, the characteristic polynomial of the submatrix
$\ 2I_{b-1}+\widetilde{C}_{(i)}(q_1,\ldots,q_k)$ corresponds to the product $\pi (q_1,\ldots ,q_i;\lambda)\pi(q_{i+1},\ldots, q_k;\lambda)$.
This means that
$$
\det \left(2I_{b-1}+\widetilde{C}_{(i)}(q_1,\ldots, q_k)-\lambda I_{b}\right)=\pi(q_1,\ldots,q_i;\lambda )\pi (q_{i+1},\ldots,q_k;\lambda).
$$
Therefore,
{\small
\begin{eqnarray*}
\frac{d}{d\lambda}\left(\det \left(2I_{b-1}+\widetilde{C}_{(i)}(q_1,\ldots,q_k)-\lambda I_b\right) \right) &=&
                                                             \pi^{\prime} (q_1,\ldots,q_i,\lambda)\pi(q_{i+1},\ldots,q_k;\lambda)+\\
                                               & & \pi(q_1,\ldots,q_i;\lambda )\pi\prime(q_{i+1},\ldots,q_k;\lambda )
\end{eqnarray*}}
and hence
{\small
\begin{eqnarray*}
\left.\frac{d}{d\lambda }\left(\det\left(2I_{b-1}+\widetilde{C}_{(i)}(q_1,\ldots,q_k)-\lambda I_b\right)\right)\right\vert_{\lambda =0} &=&
                                                                                  \pi^{\prime} (q_1,\ldots ,q_i;0)\pi(q_{i+1},\ldots,q_k;0)+\\
                                               & & \pi(q_1,\ldots,q_i;0)\pi^{\prime}(q_{i+1},\ldots,q_k;0).
\end{eqnarray*}}
Using \eqref{inv_trace} we get%
\begin{equation*}
\text{tr}((2I_{b-1}+\widetilde{C}_{(i)}(q_{1},\ldots ,q_{k}))^{-1}=-\frac{\pi^{\prime}(q_{1},\ldots ,q_{i};0)}{\pi (q_{1},\ldots ,q_{i};0)}
                                                                   -\frac{\pi^{\prime}(q_{i+1},\ldots ,q_{k};0)}{\pi (q_{i+1},\ldots ,q_{k};0)}
\end{equation*}%
and using \eqref{identity}, we obtain%
{\small
\begin{equation}
\text{tr}((2I_{b-1}+\widetilde{C}_{(i)}(q_{1},\ldots ,q_{k}))^{-1}= -\frac{p\prime (q_{1},\ldots ,q_{i};-2)}{p(q_{1},\ldots ,q_{i};-2)}-\frac{%
p\prime (q_{i+1},\ldots ,q_{k};-2)}{p(q_{i+1},\ldots ,q_{k};-2)}.\label{sum}
\end{equation}}

\begin{example}
Let us determine an upper bound on the algebraic connectivity of the caterpillar $T\left(4,9,0,1\right)$ of Example~\ref{ex2} (which is
$\mu(4,9,0,1)=0.1862$), taking into account Theorem~\ref{lb_and_ub} and determining $\text{tr}((2I_{b-1}+\widetilde{C}_{(2)}(4,9,0,1))^{-1})$.\\
From \eqref{sum}, it follows that
$\text{tr}\left((2I_{b-1}+\widetilde{C}_{(2)}(4,9,0,1))^{-1}\right)= -\frac{p^{\prime}(4,9;-2)}{p(4,9;-2)}-\frac{p^{\prime}(0,1;-2)}{p(0,1;-2)}.$
Since $-\frac{p^{\prime}(4,9;-2)}{p(4,9;-2)}=\frac{67}{15}=4.4667$ and $-\frac{p^{\prime}(0,1;-2)}{p(0,1;-2)}=\frac{11}{6}=1.8333,$
then
\begin{equation*}
\text{tr}\left((2I_{b-1}+\widetilde{C}_{(i)}(q_1,\ldots,q_k))^{-1}\right) = 4.4667 + 1.8333 = 6.300.
\end{equation*}%
From Example~\ref{ex_lb_robbiano}, we know that $\text{tr}\left((2I_{b}+C(4,9,0,1))^{-1}\right) = \frac{382}{36} = 10.6111$, and thus
the upper bound on the algebraic connectivity of $T\left(4,9,0,1\right),$ obtained by Theorem~\ref{lb_and_ub}, is not greater
than
\begin{eqnarray*}
\frac{1}{\text{tr}\left((2I_{b}+C(4,9,0,1))^{-1}\right)-\text{tr}\left((2I_{b-1}+\widetilde{C}_{(i)}(4,9,0,1))^{-1}\right)}& = &
                                                                                                             \frac{1}{10.6111 - 6.300}\\
                                                                                                                           & = & 0.2320.
\end{eqnarray*}
\end{example}

\subsection{A few computational experiments}
In this subsection a few computational experiments are presented, with the obtained lower and upper bounds
on the algebraic connectivity of caterpillars, comparing the distinct introduced upper bounds to each other.

In the next table, the upper bound \eqref{ub_1}, and the upper and lower bounds, (ub) and (lb), of Theorem~\ref{lb_and_ub}
are presented for several caterpillars $T(q_1, q_2, \ldots, q_k)$. From these results, it seems that the lower
and upper bounds of Theorem~\ref{lb_and_ub} are very accurate, since in most of the cases they are close to each other.

\begin{center}
\begin{tabular}{|l|c|c|c|c|}\hline
$(q_1,q_2,\ldots,q_{k-1},q_k)$& $\mu(q_1,\ldots,q_k)$& \eqref{ub_1}&   (ub)   &    (lb)   \\ \hline
$(3,2,1,0,5,4)$               & $0.0601$             & $0.2788$    & $0.0658$ &  $0.0372$ \\
$(2,0,3,4,7)$                 & $0.0893$             & $0.2536$    & $0.1056$ & $0.0514$  \\
$(3,5,0,0, 9,10)$             & $0.0398$             & $0.3087$    & $0,0423$ & $0.0270$  \\
$(9,5,5,4,2,0,3)$             & $0.0407$             & $0.2157$    & $0.0500$ & $0.0290$  \\
$(5,0,5,0,5,0,5,0,5)$         & $0.0285$             & $1.0000$    & $0.0346$ & $0.0167$  \\
$(3,9,10,0,5,0,4,2,0,7)$      & $0.0173$             & $0.1624$    & $0.0201$ & $0.0108$  \\ \hline
\end{tabular}
\end{center}

{\bf Acknowledgements}.
Domingos M. Cardoso and Enide A. Martins supported by {\it FEDER} founds through {\it COMPETE}--Operational
Programme Factors of Competitiveness and by Portuguese founds through the {\it Center for Research and
Development in Mathematics and Applications} (University of Aveiro) and the Portuguese Foundation for Science
and Technology (``FCT--Funda\c{c}\~{a}o para a Ci\^{e}ncia e a Tecnologia''), within project PEst-C/MAT/UI4106/2011
with COMPETE number FCOMP-01-0124-FEDER-022690. M. Robbiano partially supported by Fondecyt Grant 11090211,
Chile. All the authors are members of the Project PTDC/MAT/112276/2009.

\end{document}